\documentclass[english]{amsart}

\usepackage{babel}
\usepackage{amsmath}
\usepackage{amsthm}
\usepackage{amssymb}
\usepackage[utf8]{inputenc}
\usepackage{graphicx}
\usepackage[backend=biber,style=numeric,citestyle=numeric,sorting=nyt,maxnames=99,minnames=99]{biblatex}
\usepackage{xurl}
\AtBeginBibliography{\sloppy}
\usepackage{ifthen}
\usepackage{csquotes}
\usepackage[colorlinks = true, urlcolor = blue, linkcolor = blue, citecolor = blue, breaklinks = true]{hyperref}
\usepackage{breakcites}
\usepackage{cleveref}
\usepackage[dvipsnames]{xcolor}
\usepackage{tikz-cd}
\usepackage{quiver}
\usetikzlibrary{spath3, intersections}
\usetikzlibrary{
  knots,
  hobby,
  decorations.pathreplacing,
  shapes.geometric,
  calc
}

\tikzset{
  knot diagram/every strand/.append style={
    ultra thick,
    red
  },
  show curve controls/.style={
    postaction=decorate,
    decoration={show path construction,
      curveto code={
        \draw [blue, dashed]
        (\tikzinputsegmentfirst) -- (\tikzinputsegmentsupporta)
        node [at end, draw, solid, red, inner sep=2pt]{};
        \draw [blue, dashed]
        (\tikzinputsegmentsupportb) -- (\tikzinputsegmentlast)
        node [at start, draw, solid, red, inner sep=2pt]{}
        node [at end, fill, blue, ellipse, inner sep=2pt]{}
        ;
      }
    }
  },
  show curve endpoints/.style={
    postaction=decorate,
    decoration={show path construction,
      curveto code={
        \node [fill, blue, ellipse, inner sep=2pt] at (\tikzinputsegmentlast) {}
        ;
      }
    }
  }
}
\usepackage{url}
\usepackage{dcolumn}
\usepackage{xcolor}
\usepackage{tikz}
\usetikzlibrary{decorations.markings}
\usepackage{mathtools}
\usepackage{nicematrix}
\usepackage{mathrsfs}
\usepackage{booktabs}
\usepackage{pdfpages}
\usepackage{bbm}
\usepackage{enumitem}
\usepackage{scalerel}
\usepackage[left=3cm,right=3cm]{geometry}
\usepackage{stackengine,wasysym}

\usetikzlibrary{matrix,arrows,chains}

\newtheorem{prop}{Proposition}[section]
\newtheorem{thm}[prop]{Theorem}
\newtheorem*{theorem*}{Theorem}
\newtheorem{cor}[prop]{Corollary}
\newtheorem{lem}[prop]{Lemma}

\theoremstyle{definition}
\newtheorem{rem}[prop]{Remark}

\newtheorem{de}[prop]{Definition}
\newtheorem{ex}[prop]{Example}

\newtheorem{con}[prop]{Construction}

\newtheorem*{quest*}{Question}
\newtheorem*{cor*}{Corollary}

\newcommand\reallywidetilde[1]{\ThisStyle{%
  \setbox0=\hbox{$\SavedStyle#1$}%
  \stackengine{-.1\LMpt}{$\SavedStyle#1$}{%
    \stretchto{\scaleto{\SavedStyle\mkern.2mu\AC}{.5150\wd0}}{.4\ht0}%
  }{O}{c}{F}{T}{S}%
}}

\makeatletter
\tikzcdset{
  open/.code     = {\tikzcdset{hook, circled};},
  closed/.code   = {\tikzcdset{hook, slashed};},
  open'/.code    = {\tikzcdset{hook', circled};},
  closed'/.code  = {\tikzcdset{hook', slashed};},
  circled/.code  = {\tikzcdset{markwith = {\draw (0,0) circle (.375ex);}};},
  slashed/.code  = {\tikzcdset{markwith = {\draw[-] (-.0ex,-.4ex) -- (.0ex,.4ex);}};},
  markwith/.code ={
    \pgfutil@ifundefined%
    {tikz@library@decorations.markings@loaded}%
    {\pgfutil@packageerror{tikz-cd}{You need to say %
      \string\usetikzlibrary{decorations.markings} to use arrows with markings}{}}{}%
    \pgfkeysalso{/tikz/postaction = {
      /tikz/decorate,
      /tikz/decoration={markings, mark = at position 0.5 with {#1}}}
    }
  },
}
\makeatother

\DeclareMathOperator{\Z}{\mathbb{Z}}

\DeclareMathOperator{\C}{\mathbb{C}}
\DeclareMathOperator{\A}{\mathbb{A}}
\DeclareMathOperator{\PP}{\mathbb{P}}
\DeclareMathOperator{\T}{\mathcal{T}}
\renewcommand{\P}{\PP}
\DeclareMathOperator{\G}{\mathbb{G}}
\DeclareMathOperator{\Set}{Set}

\DeclareMathOperator{\id}{id}

\DeclareMathOperator{\Sh}{Sh}

\DeclareMathOperator{\Hom}{Hom}

\newcommand{\op}{{\mathrm{op}}}

\renewcommand{\lim}{\limi}
\DeclareMathOperator{\limi}{lim}

\DeclareMathOperator{\copro}{\scaleobj{0.8}{\amalg}}
\DeclareMathOperator{\const}{const}

\DeclareMathOperator{\Mod}{Mod}

\DeclareMathOperator{\Eq}{Eq}
\DeclareMathOperator{\Top}{Top}

\DeclareMathOperator{\Ho}{Ho}

\DeclareMathOperator{\Spc}{Spc}
\newcommand{\Cov}[1][]{\ifthenelse{\equal{#1}{}}
    {\mathrm{Cov}}
    {#1\text{-}\mathrm{Cov}}
}
\DeclareMathOperator{\tCov}{{\mathrm{tCov}}}

\DeclareMathOperator{\sSet}{sSet}

\DeclareMathOperator{\Sing}{Sing}

\DeclareMathOperator{\Map}{Map}

\DeclareMathOperator{\Sm}{Sm}

\DeclareMathOperator{\Spec}{Spec}

\newcommand{\et}{{\mathrm{\acute{e}t}}}

\DeclareMathOperator{\Ao}{\mathbb{A}^1}

\DeclareMathOperator{\Shv}{Sh}
\DeclareMathOperator{\PShv}{PSh}

\DeclareMathOperator{\fib}{fib}

\DeclareMathOperator{\mot}{mot}
\newcommand{\Lmot}[1][]{\ifthenelse{\equal{#1}{}}
    {L_{\mathrm{mot}}}
    {L_{\mathrm{mot}, #1}}
}

\DeclareFieldFormat{postnote}{#1}
\DeclareFieldFormat{multipostnote}{#1}

\graphicspath{ {./Bilder/} }

\tikzset{double line with arrow/.style args={#1,#2}{decorate,decoration={markings,%
mark=at position 0 with {\coordinate (ta-base-1) at (0,1pt);
\coordinate (ta-base-2) at (0,-1pt);},
mark=at position 1 with {\draw[#1] (ta-base-1) -- (0,1pt);
\draw[#2] (ta-base-2) -- (0,-1pt);
}}}}

\tikzset{%
    symbol/.style={%
        draw=none,
        every to/.append style={%
            edge node={node [sloped, allow upside down, auto=false]{$#1$}}}
    }
}

\tikzset{
  open/.style = {decoration = {markings, mark = at position 0.5 with { \node[transform shape, scale = .8] {$\circ$}; } }, postaction = {decorate} }
}

\begin{document}

\title{On the notion of $\A^1$-coverings}

\author[Giustetto]{Marco Giustetto}
\address[Giustetto]{
Department of Mathematics, Universit\"{a}t Osnabr\"{u}ck, 
Albrechtstrasse 28a, 49076 Osnabr\"{u}ck,
Germany}
\email{marco.giustetto@uni-osnabrueck.de}

\author[Wittich]{Thor Wittich}
\address[Wittich]{
Department of Mathematics, Universit\"{a}t Osnabr\"{u}ck, 
Albrechtstrasse 28a, 49076 Osnabr\"{u}ck,
Germany}
\email{thor.wittich@uni-osnabrueck.de}

\begin{abstract}
We introduce the notion of $\A^1$-$n$-coverings and begin the development of their theory over an arbitrary base scheme $S$, in analogy to that of higher coverings in an $\infty$-topos. We show that for $n=0$, our notion recovers Morel's $\A^1$-coverings between simplicial sheaves. Under suitable assumptions, we then provide a Galois-theoretic classification of $\A^1$-$n$-coverings. Finally, we answer a question of Morel concerning a geometric criterion for $\A^1$-coverings in the negative.

\end{abstract}
\maketitle

{ \hypersetup{hidelinks} \tableofcontents }

\section{Introduction}

Covering spaces are a fundamental gadget in various flavors of geometry, such as algebraic topology or algebraic geometry. For a well-behaved base space or scheme $X$, these can be classified in terms of $\pi_1(X)$-sets, where $\pi_1(X)$ stands for the respective notions of fundamental groups. Therefore, covering spaces provide a geometric interpretation of fundamental groups, serve as means of organizing local systems, and have found ubiquitous uses.

\hfill

In \cite{morel}, Morel developed the basic theory of $\A^1$-coverings for simplicial Nisnevich sheaves on $\Sm_k$, where $k$ is a perfect base field of characteristic not $2$. Taking inspiration from the unique homotopy lifting theorem in topology, these coverings are defined by having the unique right lifting property against acyclic $\A^1$-cofibrations or, equivalently, against $\A^1$-equivalences. For example, the $\G_m$-torsor $\A^{n+1} \setminus \{ 0 \} \twoheadrightarrow \P^n$ is shown to be the universal $\A^1$-covering for all $n \geq 2$, yielding $\underline{\pi}_1^{\A^1}(\P^n) \cong \G_m$; see Theorem 7.13 of \cite{morel}. This theory was for instance used by Wendt and the second-named author in their article on algebraic knots \cite{ww} to define a novel rectification invariant.

\hfill

In this article, we study $\A^1$-coverings over more general base schemes $S$ and define their higher analogues. To this end, we first recall the notion of $n$-coverings in $\infty$-topoi, which are defined as $n$-truncated maps, where $0 \leq n \leq \infty$. This notion was studied by Constantin \cite{high_cov} in more detail, whose preprint appeared while this work was in preparation. We then define $\A^1$-$n$-coverings as $n$-truncated maps which are right orthogonal to $\A^1$-equivalences. Our first main result identifies $\A^1$-$0$-coverings with Morel's aforementioned notion of $\A^1$-coverings.
\begin{theorem*}[Theorem \ref{morel_a1_cov}]
Let $S$ be a qcqs base scheme of finite dimension. For a simplicial Nisnevich sheaf $B$ on $\Sm_S$, there is an equivalence of $1$-categories
$$\Cov_{\A^1}(B) \xrightarrow{\sim} \Cov[0]_{\A^1}(L_W(B))$$
where $L_W \colon \Shv(\Sm_S,\sSet) \to \Shv(\Sm_S)$ is the Dwyer-Kan localization functor to the underlying $\infty$-topos.
\end{theorem*} 
This comparison result justifies our definition as an extension of the existing theory and enables us to interpret Morel's $\A^1$-coverings intrinsically in the $\infty$-topos $\Shv(\Sm_S)$. Along the way we also establish that $\A^1$-$n$-coverings only depend on the motivic homotopy type of the base $B$; see Corollary \ref{middle_eq}. While this article was in preparation, this result was independently obtained for $n = 0$ by Maity in Section 3.4 of \cite{maity}, who also proposed $\A^1$-$0$-coverings as an $\infty$-categorical variant of Morel's notion.

\hfill

In Remark 7.10 of \cite{morel}, Morel stated a classification of $\A^1$-coverings of a pointed $\A^1$-connected simplicial sheaf $X$ on $\Sm_k$ in terms of $\A^1$-invariant sheaves of sets with an action of $\underline{\pi}_1^{\A^1}(X)$, which was written out by Asok and Doran \cite[Theorem 3.16]{asokdoran} in their article on $\A^1$-excision. Our next main result extends this classification to general base schemes and simultaneously to higher $\A^1$-coverings.
\begin{theorem*}[Theorem \ref{classification_of_n_coverings}]
Let $0 \leq n \leq \infty$ and let $B$ be a Nisnevich sheaf on $\Sm_S$ for any base scheme $S$. Suppose that there exists a section $s$ of the canonical map $\Lmot(B) \to \underline{\pi}_0^{\A^1}(B)$ and that $\underline{\pi}_0^{\A^1}(B)$ is $\A^1$-invariant. Then applying the functor $\Lmot$ pointwise and taking pullbacks along $s$ afterwards induces an equivalence of $(n+1)$-categories
    \[\Cov[n]_{\Ao}(B) \xrightarrow{\sim} \Mod_{(\Omega_s \Lmot(B))_{\leq n}}\!\left(\left(\Sh_{\A^1}(\Sm_S)_{\leq n}\right)_{/\underline{\pi}_0^{\A^1}(B)}\right),\]
    where $\Omega_s$ denotes the relative loop space in the Nisnevich $\infty$-topos.
\end{theorem*}
So after choosing the indicated section, $\A^1$-$n$-coverings are classified by $n$-truncated $\A^1$-invariant objects equipped with an action of the relative loop object. This result follows from a corresponding classification of $n$-coverings in an $\infty$-topos together with a study of the interplay between the truncatedness and the orthogonality defining the higher $\A^1$-coverings. Morel also hinted in Remark 7.10 of \cite{morel} at a fully non-pointed generalization for $n = 0$ by means of an $\A^1$-fundamental groupoid, but such a variant remains unclear to us. In topology, the groupoid version is established by reducing to components, but in general $\infty$-topoi or in $\A^1$-homotopy theory, there is not necessarily a meaningful notion of decomposition of an object into its components. 

\hfill

A further question asked by Morel in Remark 7.3 of \cite{morel} concerns the potential geometric nature of $\A^1$-coverings between smooth schemes. In more detail, Morel asked whether a morphism $p \colon X \to Y$ between smooth schemes is an $\A^1$-covering if and only if it has the unique right lifting property against zero sections $0_U \colon U \to \A^1 \times \, U$ for all smooth schemes $U$. With our final main result, we answer this question in the negative:
\begin{theorem*}[Theorem \ref{counterex to Morel}]
Let $\C$ be the base field. There exists a smooth separated scheme $T$ together with a rational point $p_1 \colon \Spec(\C) \to T$ such that:
\begin{enumerate}
\item[(i)] The map $p_1$ is right orthogonal to zero sections $0_U \colon U \to \A^1 \times \, U$ for all smooth schemes $U$.
\item[(ii)] The map $p_1$ is not an $\A^1$-covering.
\end{enumerate}
\end{theorem*} 
Since this failure already occurs for $\A^1$-coverings, that is $\A^1$-$0$-coverings, this also rules out a direct extension of Morel's proposed geometric criterion for higher $\A^1$-coverings. Let us also note that another starting point for a more geometric understanding of $\A^1$-coverings is studied in Section 5 of \cite{asokdoran}, where the authors study the relation of quotients of free actions of suitable algebraic groups to $\A^1$-covering spaces.

The example above builds on Construction 4.3 of Balwe–Hogadi–Sawant \cite{BHS}. Any algebraic homotopy in $T$ starting at the point $p_1$ turns out to be constant in any smooth family, but there exists a so-called $\A^1$-ghost homotopy supplying an equality of the form
$$[p_1] = [q] \in \underline{\pi}_0^{\A^1}(T)(\C)$$
for a second rational point $q$ of $T$. Using this, we construct a lifting problem against an $\A^1$-equivalence without a lift. Our result hence shows that the two notions defined by right lifting properties behave differently with respect to algebraic homotopies and $\A^1$-ghost homotopies in general. It remains an open question whether it is possible to test $\A^1$-coverings against a larger class of $\A^1$-equivalences of geometric nature than the class of zero sections. We can show, however, that Morel's suggested geometric criterion works on the level of presheaves; see Corollary \ref{morels_question}.

\subsection{Notation and Conventions} Throughout the article we let $S$ denote a general base scheme and we emphasize if we require further properties of $S$. For the convenience of the reader, we provide a table listing some of our notations and conventions:
\begin{center}
    \begin{tabular}{cl}
        $n$-category & $(n, 1)$-category, i.e., an $\infty$-category with $(n-1)$-truncated  mapping spaces \\
        $\Sm_S$ & $1$-category of smooth schemes over $S$ equipped with the Nisnevich topology \\
        $\sSet$ & model $1$-category of simplicial sets with the Quillen model structure\\
        $\Spc$ & $\infty$-category of spaces/homotopy types/$\infty$-groupoids \\
        $\PShv(\mathcal{G})$ & $\infty$-category of presheaves on a site $\mathcal{G}$ \\
        $\Shv(\mathcal{G})$ & $\infty$-category of sheaves on a site $\mathcal{G}$ \\
        $\Shv_{\A^1}(\Sm_S)$ & $\infty$-category of $\A^1$-invariant (Nisnevich) sheaves \\
        $n$-$\Cov(B)$ & $(n+1)$-category of $n$-coverings of $B$ \\
        $n$-$\Cov_{\A^1}(B)$ & $(n+1)$-category of $\A^1$-$n$-coverings of $B$ \\
        $L_{\mot}$ & motivic localization functor \\ 
        right/left orthogonal & right/left lifting property with contractible space of fillers \\
        $n$-truncated & right orthogonal to the maps $S^{n+1} \times X \to X$ for all objects $X$ \\
        $n$-connected & left orthogonal to all $n$-truncated maps \\
        saturated & contains equivalences and is closed under composition and colimits \\
        strongly saturated & saturated and satisfies 2-out-of-3 \\
        1st-of-3 for $C$ & for a class of maps $C$, whenever $g, g \circ \hspace{-1pt}f \in C$, also $f \in C$
    \end{tabular}
\end{center}
Moreover, we will often say presheaf or Nisnevich sheaf without explicitly specifying the site or the target category, which for us will mean (pre-)sheaves of spaces on $\Sm_S$.
\subsection{Acknowledgments} We wish to thank Peter Arndt, Muhammad Awais, Nidhi Gupta, Claudius Heyer, Georg Lehner, Cary Malkiewich, Oliver Röndigs, Markus Spitzweck, Matthias Wendt and Markus Zetto for useful comments and discussions. Moreover, we thank Ibrahim from Food Station Osnabrück for providing plenty of nourishing meals while the authors worked on this article. Finally, the first-named author was partially supported by the National Science Foundation under Grant No. DMS-2424139 while he was in residence at the Simons Laufer Mathematical Sciences Institute in Berkeley, California.

\section{Coverings in \texorpdfstring{$\infty$}{infty}-Topoi and in Model 1-Categories}
In this section we give a brief introduction to higher coverings in $\infty$-topoi following \cite{high_cov}. In particular, we state their Galois-theoretic classification, and in the case of spaces/$\infty$-groupoids, compare them to their classical counterpart for topological spaces. Finally, we introduce and study coverings in model $1$-categories in preparation for the next section on coverings in motivic homotopy theory.

\hfill

We begin with the notion of higher coverings as introduced and studied by Constantin \cite{high_cov}. Before we state the definition, let us note that our preferred convention differs by a shift of $1$ from the one in loc.\ cit.
\begin{de}\label{topos_cov}
    Let $B$ be an object of an $\infty$-topos $\mathcal{T}$ and let $0 \leq n \leq \infty$. An $n$-covering of $B$ is an $n$-truncated map $E \to B$ in $\T$. We denote the full subcategory of $\T_{/ B}$ spanned by $n$-coverings of $B$ by $\Cov[n](B)$.
\end{de}
It is not difficult to observe that all of the mapping spaces between objects in $\Cov[n](B)$ are $n$-truncated, so that $\Cov[n](B)$ is an $(n+1)$-category. 
\begin{rem}\label{detect_trunc}
    Let $B \to P$ be a $0$-connected map in an $\infty$-topos $\T$ and suppose it admits a section $s$. Then $s$ is $(-1)$-connected by \cite[Proposition 8.8 (3)]{rezk} and hence an effective epimorphism. Therefore, \cite[Proposition 7.10]{rezk} yields that a map $E \to B$ is $n$-truncated if and only if its base change along $s$ is.
\end{rem}
Now we state a variant of Constantin's Galois-theoretic classification of higher coverings from \cite{high_cov}. 
\begin{thm}\label{class_cov_nconn}
    Let $B$ be an object of an $\infty$-topos $\T$ and let $0 \leq n \leq \infty$. Suppose there exists a section $s$ of the canonical map $B \to B_{\leq 0}$. Then taking pullbacks along $s$ induces an equivalence of $(n+1)$-categories
    \[\Cov[n](B) \xrightarrow{\sim} \Mod_{(\Omega_s B)_{\leq n}}((\T_{\leq n})_{/B_{\leq 0}}).\]
    The inverse functor is given by taking a module $M$, extending the module structure to one over $\Omega_s B$ by adjunction, and sending the latter to the associated quotient stack $M // \Omega_s B$, i.e., the geometric realization of the simplicial diagram witnessing the module structure on $M$. This object comes together with a natural map to $* // \Omega_s B \simeq  B$.
\end{thm}
\begin{proof}
    After unpacking all of the definitions, we get an equivalence
    \[\Cov[n](B) \xrightarrow{\sim} \Mod_{(\Omega_s B)_{\leq n}}((\T_{/B_{\leq 0}})_{\leq n})\]
    from \cite[Theorem 5.6]{high_cov} applied to the slice topos $\T_{/B_{\leq 0}}$, in which $B \to B_{\leq 0}$ is a pointed 0-connected object. From truncation estimates on the forgetful functor $\T_{/B_{\leq 0}} \to \T$ we obtain an equivalence
    \[\left(\T_{/B_{\leq 0}}\right)_{\leq n} \cong \left(\T_{\leq n}\right)_{/B_{\leq 0}}\]
    which yields the claim.
\end{proof}
We will spend the rest of the section connecting the $\infty$-categorical notions from above to more classical notions of coverings. Here we note that most of the following statements are folklore. 

\hfill

Depending on the source, the classical notion of coverings for topological spaces, which we call topological coverings below, excludes certain kinds of examples. To prevent confusion, let us emphasize that we allow the source of such a covering to be empty or disconnected. Moreover, we denote by $\Top$ the $1$-category of locally path-connected, semi-locally simply connected topological spaces, and by $\tCov(B)$ the $1$-category of topological coverings of an object $B \in \Top$.
\begin{prop}\label{top_cov}
    Let $B \in \Top$. The functor $L_{\mathrm{whe}}\colon \Top \to \Spc$ which sends a topological space to the associated space/$\infty$-groupoid induces an equivalence of 1-categories
    $$L_{\mathrm{whe}}\colon\tCov(B) \xrightarrow{\sim} \Cov[0](L_{\mathrm{whe}}(B)).$$
\end{prop}
\begin{proof}
    The statement is clear if $B$ is empty. Thus we may assume $B$ to be non-empty. By decomposing $B$ into its connected components and by choosing a base point for each component, we may furthermore assume that $B$ is pointed and connected. Now we consider the topological fiber functor
    \[\mathrm{tfib}\colon \tCov(B) \to \pi_1(B)\text{-}\Mod(\Set),\]
    which sends a topological covering of $B$ to its fiber over the base point of $B$ together with the monodromy action. Since topological coverings are Hurewicz fibrations, this fiber is also a homotopy fiber. Therefore, the functor 
    $$L_{\mathrm{whe}}\colon\tCov(B) \to \Cov[0](L_{\mathrm{whe}}(B))$$
    is well-defined. Indeed, a topological covering has discrete (homotopy) fibers, which means that its image under $L_{\mathrm{whe}}$ is 0-truncated. Moreover, the monodromy action on the fiber of a topological covering of $B$ coincides with the action of $(\Omega B)_{\leq 0}$ on its homotopy fiber. In other words, the triangle
    \[\begin{tikzcd}
    	& {\Cov[0](L_{\mathrm{whe}}(B))} & \\
    	{\tCov(B)} && {\pi_1(B)\text{-}\Mod\left(\Set\right)}
    	\arrow["\fib", from=1-2, to=2-3]
    	\arrow["{L_{\mathrm{whe}}}", from=2-1, to=1-2]
    	\arrow["{\mathrm{tfib}}"', from=2-1, to=2-3]
    \end{tikzcd}\]
    commutes. The horizontal functor is an equivalence by classical covering theory (see, e.g., \cite[Theorem 3.3.2]{tomdieck}), and the top right functor is an equivalence by Theorem \ref{class_cov_nconn}. Hence the top left functor must be an equivalence as well, which proves the 
    statement.
\end{proof}
This proposition, which is also hinted at in Example 7.2 of \cite{rezk}, allows us to view topological coverings within the $\infty$-category of spaces and hence justifies the $\infty$-topos-theoretic notion of coverings.

\hfill

Next we introduce coverings in the framework of model $1$-categories. The following definition was already proposed for specific model structures, for instance in \cite{gabriel_zisman} and \cite{morel}.
\begin{de}\label{gen_covering}
    Let $\mathcal{C}$ be a model 1-category with weak equivalences $W$. We say that a map $E \to B$ in $\mathcal{C}$ is a $W$-covering if it is right orthogonal to the class $W$. Given an object $B \in \mathcal{C}$, we denote by $\Cov_W(B)$ the full subcategory of $\mathcal{C}_{/B}$ spanned by $W$-coverings.
\end{de}

Note that, if the base is the terminal object, a $W$-covering $E \to *$ is nothing more than a $W$-local object. The following proposition allows us in greater generality to rephrase orthogonality questions as locality questions, which are better studied and understood.
\begin{prop}
    Let $\mathcal{C}$ be an $\infty$-category and let $f \colon A \to B$ and $g \colon X \to Y$ be maps in $\mathcal{C}$. The following are equivalent:
    \begin{enumerate}
        \item[(i)] The map $g$ is right orthogonal to $f$.
        \item[(ii)] For any map $b \colon B \to Y$ viewed as an object in $\mathcal{C}_{/Y}$, the object $g$ is $f$-local in $\mathcal{C}_{/Y}$, where $f$ is viewed as a morphism $b \circ \!f \to b$ in $\mathcal{C}_{/Y}$.
    \end{enumerate}
\end{prop}
\begin{proof}
    Choosing a map $b \colon B \to Y$ corresponds to choosing a point in the mapping space $\Map_{\mathcal{C}}(B, Y)$ and thus in $\Map_{\mathcal{C}}(A, Y)$ by composition with $f$. For such a choice we have a morphism of fiber sequences
    \[\begin{tikzcd}
    	{\Map_{\mathcal{C}_{/Y}}(b,g)} & {\Map_{\mathcal{C}}(B,X)} & {\Map_{\mathcal{C}}(B, Y)} \\
    	{\Map_{\mathcal{C}_{/Y}}(b \circ \! f, g)} & {\Map_{\mathcal{C}}(A,X)} & {\Map_{\mathcal{C}}(A,Y)}
    	\arrow[from=1-1, to=1-2]
    	\arrow["{f^*}"', from=1-1, to=2-1]
    	\arrow["{g_*}", from=1-2, to=1-3]
    	\arrow["{f^*}"', from=1-2, to=2-2]
    	\arrow["{f^*}", from=1-3, to=2-3]
    	\arrow[from=2-1, to=2-2]
    	\arrow["{g_*}"', from=2-2, to=2-3]
    \end{tikzcd}\]
    Now (i) is equivalent to the square on the right being a Cartesian square of spaces, which is equivalent to the left vertical map between the fibers being an equivalence for every point in $\Map_{\mathcal{C}}(B, Y)$. As the latter is a reformulation of (ii), we are done.
\end{proof}
\begin{cor}\label{ortho_to_local}
    Let $S$ be a class of maps in an arbitrary $\infty$-category $\mathcal{C}$ and fix a map $p \colon E \to B$ in $\mathcal{C}$. Let $S_B \subset \mathcal{C}_{/ B}$ be the class of maps whose image under the forgetful functor $\mathcal{C}_{/ B} \to \mathcal{C}$ lands in $S$. Then $p$ is right orthogonal to $S$ in $\mathcal{C}$ if and only if $p$ is $S_B$-local in $\mathcal{C}_{/B}$. 
\end{cor}

We can generalize Lemma 7.2 of \cite{morel} to this more abstract definition of coverings in a model 1-category.
\begin{prop}
    Let $\mathcal{C}$ be a model 1-category with weak equivalences $W$, in which every object is cofibrant. Then a map $p \colon E \to B$ is a $W$-covering if and only if it is right orthogonal to the class of acyclic cofibrations.
\end{prop}
\begin{proof}
    The ``only if'' part is clear. The other direction follows from Corollary \ref{ortho_to_local} together with Ken Brown's lemma (see, e.g., \cite[Corollary 1.10]{zbMATH06267569}) applied to the functor $\Hom_{\mathcal{C}_{/B}}(-, p)\colon \left(\mathcal{C}_{/B}\right)^\op \to \Set$, where the source is given the slice model structure inherited from $\mathcal{C}$ and the target is viewed as a category with weak equivalences given by isomorphisms.
\end{proof}

Moreover, if the model 1-category that we are working with is simplicial, there is a pleasing interaction between model-categorical coverings and simplicial ones.
\begin{prop}\label{W_and_simp_cov}
    Let $\mathcal{C}$ be a simplicial model 1-category with weak equivalences $W$, in which every object is cofibrant. Let $p \colon E \to B$ be a $W$-covering. Then the simplicial set $\underline{\Hom}_{\mathcal{C}_{/B}}(f, p)$ is discrete for any map $f \colon X \to B$.
\end{prop}
\begin{proof}
    The internal mapping simplicial set in the statement is by definition the fiber over $f$ of the map of simplicial sets $p_* \colon \underline{\Hom}_{\mathcal{C}}(X, E) \to \underline{\Hom}_{\mathcal{C}}(X, B)$; see, for instance, Proposition 3.6.20 and the proof of Proposition 3.6.21 from \cite{arXiv:1310.7930}. To conclude it is hence enough to show that $p_*$ is a simplicial covering, i.e., that $p_*$ is right orthogonal to all weak equivalences of simplicial sets. By adjunction, this is equivalent to showing that $p$ is right orthogonal to all maps of the form $w \otimes \id_X$, where $w$ is a weak equivalence of simplicial sets. The definition of simplicial model categories together with Ken Brown's lemma implies that any such map is in $W$. As $p$ is a $W$-covering, we are therefore done.
\end{proof}

In the final part of this section we show that the theory of model-categorical coverings can be completely recovered from the underlying $\infty$-category, at least in the cases of interest to us. 

\begin{prop}\label{simp_cov}
    Let $B$ be a simplicial set and let $L_{\mathrm{whe}}\colon \sSet \to \Spc$ be the Dwyer-Kan localization of simplicial sets at weak equivalences. Then $L_{\mathrm{whe}}$ induces an equivalence of $1$-categories
    \[\Cov_\Delta(B) \xrightarrow{\sim} \Cov[0](L_{\mathrm{whe}}(B))\]
    where the left-hand side is the category of simplicial coverings of $B$.
\end{prop}
\begin{proof}
    This follows directly from \cite[Appendix 1, Theorem 3.2.1]{gabriel_zisman} together with Proposition \ref{top_cov}.
\end{proof}
\begin{cor}\label{trunc_sat_to_eq}
    Let $f$ be a map in the saturation of the class of weak equivalences of simplicial sets and suppose that $L_{\mathrm{whe}}(f)$ is $0$-truncated. Then $f$ is a weak equivalence.
\end{cor}
\begin{proof}
    Denote the codomain of $f$ by $B$. We write $(\sSet_{/B})_{\leq 0}$ for the full subcategory of $\sSet_{/B}$ spanned by maps $p \colon E \to B$ for which $L_{\mathrm{whe}}(p)$ is $0$-truncated. Furthermore, we denote the equivalence provided by Proposition \ref{simp_cov} by $F$ and we denote by $G$ its inverse. The statement follows if we show that the diagram
    \begin{center}
    \begin{tikzcd}
	    {(\sSet_{/B})_{\leq 0}} & \\
	    {\Cov_\Delta(B)} & {\Cov[0](L_{\mathrm{whe}}(B))}
	    \arrow["L"', from=1-1, to=2-1]
	    \arrow["{L_{\mathrm{whe}}}", from=1-1, to=2-2]
	    \arrow["G", from=2-2, to=2-1]
    \end{tikzcd}  
    \end{center}
    commutes, where $L$ is the Bousfield localization from this subcategory onto the subcategory of simplicial coverings of $B$. Indeed, if $f$ is a map as in the statement, it is by definition inverted by $L$. The commutativity of the diagram then implies that it is inverted by $L_{\mathrm{whe}}$ as well.

    Now we show that the above diagram commutes. First we observe that the related diagram
    \[\begin{tikzcd}
	    {(\sSet_{/B})_{\leq 0}} & \\
	    {\Cov_\Delta(B)} & {\Cov[0](L_{\mathrm{whe}}(B))}
	    \arrow["{L_{\mathrm{whe}}}", from=1-1, to=2-2]
	    \arrow["R", hook, from=2-1, to=1-1]
	    \arrow["F"', from=2-1, to=2-2]
    \end{tikzcd}\]
    commutes by definition of $F$, where $R$ is the inclusion functor. Functorially in an object $p\colon E \to B$ of $(\sSet_{/B})_{\leq 0}$, we have equivalences
    \[\begin{aligned}
        \Hom_{(\sSet_{/B})_{\leq 0}}(p, R \circ G \circ L_{\mathrm{whe}}(p)) & \simeq \pi_0 L_{\mathrm{whe}}\left(\underline{\Hom}_{(\sSet_{/B})_{\leq 0}}(p, R \circ G \circ L_{\mathrm{whe}}(p))\right) \\
        &\simeq \pi_0 \Map_{\Cov[0](L_{\mathrm{whe}}(B))}(L_{\mathrm{whe}}(p), L_{\mathrm{whe}} R \circ G \circ L_{\mathrm{whe}}(p)) \\
        &\simeq \pi_0 \Map_{\Cov[0](L_{\mathrm{whe}}(B))}(L_{\mathrm{whe}}(p), F \circ G \circ L_{\mathrm{whe}}(p)) \\
        &\simeq \Map_{\Cov[0](L_{\mathrm{whe}}(B))}(L_{\mathrm{whe}}(p), L_{\mathrm{whe}}(p)),
    \end{aligned}\]
    and we can choose the identity map on the level of the last line. This gives us a natural transformation $\id \to R \circ G \circ L_{\mathrm{whe}}$. Similarly, one can verify the triangular identities which witness this transformation as the unit for an adjunction between $R$ and $G \circ L_{\mathrm{whe}}$. By uniqueness of adjoints we therefore have $G \circ L_{\mathrm{whe}} \simeq L$ as desired.
\end{proof}

Next we state a useful lemma, which can, for instance, be found as Remark 4.3.4 in \cite{arXiv:2309.01757}.
\begin{lem}\label{calc_frac_model}
Let $\mathcal{C}$ be a right proper model $1$-category with weak equivalences $W$. Then the Dwyer-Kan localization functor $L \colon \mathcal{C} \to \mathcal{C}[W^{-1}]$ induces an equivalence
$$\mathcal{C}_{/B}[W_B^{-1}] \xrightarrow{\sim} \mathcal{C}[W^{-1}]_{/L(B)}$$
for every $B \in \mathcal{C}$, where $W_B$ consists of those maps in the slice whose underlying map lies in $W$.
\end{lem}
We can now extend the comparison from Proposition \ref{simp_cov} to simplicial sheaves on a site with enough points.
\begin{prop}\label{site_simp_cov}
    Let $\mathcal{\mathcal{G}}$ be a Grothendieck site with enough points and let $\mathcal{C} = \Shv(\mathcal{G}, \sSet)$ be the 1-category of simplicial sheaves on $\mathcal{G}$ equipped with the local model structure. Let 
    $$L\colon \mathcal{C} \to \mathcal{C}[W^{-1}]$$
    be the Dwyer-Kan localization functor to the underlying $\infty$-category, which is the hypercompleted $\infty$-topos of sheaves on $\mathcal{G}$. Then, for every object $B \in \mathcal{C}$, applying $L$ pointwise yields an equivalence of 1-categories
    $$\Cov_W(B) \xrightarrow{\sim} \left(\mathcal{C}[W^{-1}]_{/L(B)}\right)_{\leq 0}.$$
\end{prop}
\begin{proof}
    First note that the functor 
    $$\Cov_W(B) \xrightarrow{\sim} \left(\mathcal{C}[W^{-1}]_{/L(B)}\right)_{\leq 0}$$
    is well-defined, since Proposition \ref{W_and_simp_cov} together with Lemma \ref{calc_frac_model} shows that any $W$-covering $p \colon E \to B$ is sent to a 0-truncated object $L(p)$ in $\mathcal{C}[W^{-1}]_{/L(B)}$. To show that it is fully faithful, we consider two $W$-coverings $p_1 \colon E_1 \to B$ and $p_2 \colon E_2 \to B$ and the span of equivalences
    \begin{center}
    \begin{tikzcd}
    \Hom_{\mathcal{C}_{/B}}(p_1, p_2) \arrow[r, "\sim"] & \Hom_{\Ho(\mathcal{C}[W^{-1}]_{/L(B)})}(L(p_1), L(p_2)) & \Map_{\mathcal{C}[W^{-1}]_{/L(B)}}(L(p_1), L(p_2)). \arrow[l, swap, "\sim"]
    \end{tikzcd}
    \end{center}
    Here the first equivalence stems from the fact that the equivalence relation defining the hom-sets in the $1$-category $\Ho(\mathcal{C}[W^{-1}]_{/L(B)})$ is the identity relation if the target is a $W$-covering. The second equivalence follows from $L(p_2)$ being 0-truncated.

    Now we tackle the essential surjectivity. For this we let $p'' \colon E'' \to L(B)$ be a 0-truncated map in $\mathcal{C}[W^{-1}]$. By Lemma \ref{calc_frac_model} together with \cite[Theorem 2.7]{mv}, there exists a morphism of simplicial sheaves $p' \colon E' \to B$ which is equivalent to $p''$ after applying the localization functor. Using the small object argument, we may factor $p'$ as a morphism $f$ in the saturated class generated by $W$ in $\mathcal{C}$ followed by a $W$-covering $p$. To conclude, it suffices to show that $f$ is inverted by the localization functor, which amounts to checking that $f$ induces weak equivalences of simplicial sets stalkwise. So let $x$ be any point of $\mathcal{G}$ and denote by $f_x$ the map induced by $f$ on stalks. As the stalk functor is cocontinuous, we know that $f_x$ belongs to the saturated class generated by weak equivalences of simplicial sets. Since $L(p')$ and $L(p)$ are $0$-truncated, truncation estimates yield that $L(f)$ is 0-truncated. So $L_{\mathrm{whe}}(f_x)$ is 0-truncated as well. Thus we can apply Corollary \ref{trunc_sat_to_eq} to conclude that $f_x$ is a weak equivalence, as desired.
\end{proof}
\begin{rem}
    Using the global model structure in place of the local one, we can get rid of the hypercompletion arising from the localization functor above. The essential surjectivity is then shown by arguing sectionwise rather than stalkwise.
\end{rem}
Finally, we now specialize to simplicial Nisnevich sheaves and $\A^1$-weak equivalences.
\begin{thm}\label{morel_a1_cov}
    Let $S$ be any base scheme and let $\mathcal{C}$ be the 1-category of simplicial Nisnevich sheaves on $\Sm_S$ equipped with the local model structure. Let $W$ be the class of local weak equivalences and let $W_{\A^1}$ be the class of $\A^1$-weak equivalences as in \cite{mv}. Furthermore, let $L \colon \mathcal{C} \to \mathcal{C}[W^{-1}]$ be the Dwyer-Kan localization functor. For any $B \in \mathcal{C}$, the equivalence of Proposition \ref{site_simp_cov} identifies the subcategory $\Cov_{W_{\A^1}}(B) \subset \Cov_{W}(B)$ with the full subcategory of $\Cov[0](L(B))$ spanned by maps that are right orthogonal to the class of $\A^1$-equivalences in $\mathcal{C}[W^{-1}]$.
\end{thm}
\begin{proof}
    Let $p \colon E \to B$ be a $W$-covering and let $W_{\A^1\!,\, B}$ (respectively $W_{\A^1\!,\, L(B)}$) be the class of maps in $\mathcal{C}_{/B}$ (respectively in $\mathcal{C}[W^{-1}]_{/L(B)}$) which get mapped to $W_{\A^1}$ (respectively, to $\A^1$-equivalences in $\mathcal{C}[W^{-1}]$) under the forgetful functor. We argue in four steps:
    \begin{enumerate}
        \item[(i)] Applying Corollary \ref{ortho_to_local} with ambient category $\mathcal{C}[W^{-1}]$, we see that $L(p)$ is an $\A^1$-fibration if and only if $\Map_{\mathcal{C}[W^{-1}]_{/L(B)}}(-, L(p))$ sends maps in $W_{\A^1\!,\, L(B)}$ to equivalences.
        \item[(ii)] We claim that (i) is equivalent to $\underline{\Hom}_{\mathcal{C}_{/B}}(-, p)$ sending maps in $W_{\A^1\!,\, B}$ to weak equivalences of simplicial sets. Indeed, if we denote by $L_{\mathrm{whe}}\colon \sSet \to \Spc$ the localization of simplicial sets at weak equivalences, then there is an equivalence of spaces
        \[L_{\mathrm{whe}}\left(\underline{\Hom}_{\mathcal{C}_{/B}}(-, p)\right) \cong \Map_{\mathcal{C}[W^{-1}]_{/L(B)}}(L(-), L(p)).\]
        By Lemma \ref{calc_frac_model} together with \cite[Theorem 2.7]{mv}, each morphism in $W_{\A^1\!,\, L(B)}$ can be lifted to a morphism in $W_{\A^1\!,\, B}$ via $L$. So if $p$ fulfills (ii), then it also fulfills (i). Conversely, $L$ sends maps in $W_{\A^1, B}$ to maps in $W_{\A^1\!,\, L(B)}$. Thus if $p$ satisfies (i), then $\underline{\Hom}_{\mathcal{C}_{/B}}(-, p)$ sends maps in $W_{\A^1\!,\, B}$ to maps of simplicial sets turned by $L_{\mathrm{whe}}$ into equivalences of spaces. This is equivalent to (ii).
        \item[(iii)] We have that (ii) is equivalent to $p$ being local in $\mathcal{C}_{/B}$ with respect to maps in $W_{\A^1\!,\, B}$. Indeed, by Proposition \ref{W_and_simp_cov}, the functor $\underline{\Hom}_{\mathcal{C}_{/B}}(-, p)$ lands in discrete simplicial sets and a weak equivalence of discrete simplicial sets is necessarily an isomorphism.
        \item[(iv)] Appealing again to Corollary \ref{ortho_to_local}, this time with ambient category $\mathcal{C}$, we conclude that (iii) holds if and only if $p$ is a $W_{\A^1}$-covering.
    \end{enumerate} 
\end{proof}

\section{Coverings in \texorpdfstring{$\A^1$}{A1}-Homotopy Theory}
In this section we turn our focus towards coverings in $\A^1$-homotopy theory. We begin by introducing higher $\A^1$-coverings and developing their basic properties. Afterwards, we provide a Galois-theoretic classification result and discuss examples. Finally, we discuss Morel's proposed geometric criterion from Remark 7.3 of \cite{morel} for $\A^1$-coverings between smooth schemes on the level of sheaves and presheaves, where we obtain different answers.
\subsection{Higher $\A^1$-coverings}
Theorem \ref{morel_a1_cov} from the previous section suggests extending the notion of $\A^1$-coverings to higher levels of truncation while keeping the right orthogonality to $\A^1$-equivalences. This leads to the following definition.
\begin{de}\label{higher_a1_cov_def}
    Let $0 \leq n \leq \infty$. A map $p\colon E \to B$ in $\Shv(\Sm_S)$ is an $\A^1$-$n$-covering if
    \begin{enumerate}
        \item[(i)] it is an $\A^1$-fibration, i.e., it is right orthogonal to the class of $\A^1$-equivalences;
        \item[(ii)] $\Lmot(p)$ is $n$-truncated.
    \end{enumerate}  
\end{de}
If $n = \infty$, this is consistent with Definition 3.4.8 of \cite{maity} and coincides with the concept of $L_{\A^1}$-fibration from Section 6 of \cite{macpherson}. We denote by $\Cov[n]_{\A^1}(B)$ the full subcategory of $\Sh(\Sm_S)_{/B}$ spanned by $\A^1$-$n$-coverings. For $n=0$, this definition was also proposed in \cite{maity} and many of the following results have been established in this case in Section 3.4 of \cite{maity}.
\begin{rem}
    If the base scheme $S$ is qcqs and finite dimensional, the $\infty$-topos $\Sh(\Sm_S)$ is hypercomplete; see Proposition A.3 of \cite{norms}. Combining Lemma 7.2 of \cite{morel} with Theorem \ref{morel_a1_cov}, we obtain an equivalence between Morel's 1-category of $\A^1$-coverings of a simplicial sheaf $B$ and $\Cov[0]_{\A^1}(L(B))$, where $L$ is the Dwyer-Kan localization from simplicial sheaves to sheaves.
\end{rem}
The rest of this subsection will be devoted to analyzing this definition. In particular, we want to show that higher $\A^1$-coverings admit a Galois-theory analogous to that of higher coverings in an $\infty$-topos. To that end, we introduce some auxiliary definitions and lemmas.
\begin{de}\label{a1_mod}
    The $\A^1$-modality $(\mathcal{L}_{\A^1}, \mathcal{R}_{\A^1})$ is the modality on $\Shv(\Sm_S)$ whose left class is generated by the projections $\A^1 \times \ X \to X$ for smooth schemes $X$ over $S$.
\end{de}
We first establish a base change property for reflective factorization systems associated to modalities.
\begin{prop}\label{equiv_bc}
    Let $(\mathcal{L}, \mathcal{R})$ be a modality on an $\infty$-topos\footnote{It actually suffices that $\T$ is a presentable category whose colimits are stable under base change.} $\T$ and let $\left(\mathcal{L}^{3/2}, \mathcal{R}^{2/3}\right)$ be the reflective factorization system on $\T$ associated to it, as in \cite[\S 1.3.3]{soa}. Then maps in $\mathcal{L}^{3/2}$ are stable under base change along maps in $\mathcal{R}^{2/3}$. Moreover, any commutative square
    \begin{equation}\label{easy_cart}
        \begin{tikzcd}
        	X & Y \\
        	A & B
        	\arrow["{l'}", from=1-1, to=1-2]
        	\arrow["r"', from=1-1, to=2-1]
        	\arrow["{r'}", from=1-2, to=2-2]
        	\arrow["l"', from=2-1, to=2-2]
        \end{tikzcd}
    \end{equation}
    in which $l,l' \in \mathcal{L}^{3/2}$ and $r, r' \in \mathcal{R}^{2/3}$ is automatically Cartesian.
\end{prop}
\begin{proof}
    Let $\mathcal{Q} \subset \mathcal{L}^{3/2}$ be the subclass of maps whose base changes along maps in $\mathcal{R}^{2/3}$ belong to $\mathcal{L}^{3/2}$. Our goal is to prove that the above inclusion is an equality. 
    
    Since $\mathcal{L}$ is stable under base change, we have $\mathcal{L} \subset \mathcal{Q}$. Moreover, $\mathcal{Q}$ contains all equivalences and is stable under composition. As colimits in $\mathcal{T}$ are stable under base change, $\mathcal{Q}$ is also stable under colimits.

    To show that $\mathcal{Q}$ is closed under 1st-of-3, we first consider a square as in (\ref{easy_cart}) where $l$ is furthermore assumed to belong to $\mathcal{Q}$. By definition, the base change of $l$ along $r'$ belongs to $\mathcal{L}^{3/2}$, which yields that the comparison map $X \to A \times_B Y$ is in $\mathcal{L}^{3/2}$. However, the latter is also in $\mathcal{R}^{2/3}$, by the stability of $\mathcal{R}^{2/3}$ under base change and 1st-of-3. Hence it is invertible, so the square is Cartesian.

    Now let $f \colon A \to B$ and $g \colon B \to C$ be maps with $g, g \circ \hspace{-1pt} f \in \mathcal{Q}$, and consider a map $r \colon X \to B$ in $\mathcal{R}^{2/3}$. Factoring the composite $g \circ r$ using the $(\mathcal{L}^{3/2},\mathcal{R}^{2/3})$ factorization system and pulling back $f$ along $r$, we obtain a commutative diagram
        \[\begin{tikzcd}
        	W & X & Y \\
        	A & B & C
        	\arrow["{f'}", from=1-1, to=1-2]
        	\arrow[from=1-1, to=2-1]
        	\arrow["{g'}", from=1-2, to=1-3]
        	\arrow["r", from=1-2, to=2-2]
        	\arrow["{r'}", from=1-3, to=2-3]
        	\arrow["f"', from=2-1, to=2-2]
        	\arrow["g"', from=2-2, to=2-3]
        \end{tikzcd}.\]
    Here we have $g' \in \mathcal{L}^{3/2}$ and $r' \in \mathcal{R}^{2/3}$. The left square is Cartesian by construction, and the right square is Cartesian by the previous paragraph. Thus $g' \circ f'$ is a base change of $g \circ \hspace{-1pt}f$ along $r'$, so it belongs to $\mathcal{L}^{3/2}$. Since $\mathcal{L}^{3/2}$ is closed under 1st-of-3, we obtain $f' \in \mathcal{L}^{3/2}$, and for the same reason we also have $f \in \mathcal{L}^{3/2}$. This shows that $f \in \mathcal{Q}$.

    We have therefore shown that $\mathcal{Q}$ is a strongly saturated class containing $\mathcal{L}$, and it is contained in its strong saturation $\mathcal{L}^{3/2}$. Hence $\mathcal{Q} = \mathcal{L}^{3/2}$, which proves the first part of the statement. The second part now follows from the argument for squares of the form (\ref{easy_cart}) above.
\end{proof}
\begin{rem}\label{a1_mod_recover}
    Note that when the modality in the above statement is the $\A^1$-modality, the resulting reflective factorization system is that of ($\A^1$-equivalences, $\A^1$-fibrations).
\end{rem}
Applied to the $\A^1$-modality, the above statement shows that higher $\A^1$-coverings can be transported along $\A^1$-equivalences of the base. What follows is hence a generalization of Lemma 7.9 of \cite{morel}.
\begin{cor}
    Let $p \colon E \to B$ be an $\A^1$-$n$-covering and let $w \colon B \to B'$ be an $\A^1$-equivalence. Up to contractible choice, this datum can be uniquely completed to a Cartesian square
    \[\begin{tikzcd}[cramped]
    	E & {E'} \\
    	B & {B'}
    	\arrow["{w'}", from=1-1, to=1-2]
    	\arrow["p"', from=1-1, to=2-1]
    	\arrow["\lrcorner"{anchor=center, pos=0.125}, draw=none, from=1-1, to=2-2]
    	\arrow["{p'}", from=1-2, to=2-2]
    	\arrow["w"', from=2-1, to=2-2]
    \end{tikzcd}\]
    where $w' \colon E \to E'$ is an $\A^1$-equivalence and $p' \colon E' \to B'$ an $\A^1$-$n$-covering.
\end{cor}
\begin{proof}
    Consider the square 
        \[\begin{tikzcd}[cramped]
    	E & {E'} \\
    	B & {B'}
    	\arrow["{w'}", from=1-1, to=1-2]
    	\arrow["p"', from=1-1, to=2-1]
    	\arrow["{p'}", from=1-2, to=2-2]
    	\arrow["w"', from=2-1, to=2-2]
    \end{tikzcd}\]given by factoring the composite $w \circ p$ into an $\A^1$-equivalence $w' \colon E \to E'$ followed by an $\A^1$-fibration $p' \colon E' \to B'$. By Proposition \ref{equiv_bc} together with Remark \ref{a1_mod_recover}, it is Cartesian. Moreover, since the horizontal arrows are $\A^1$-equivalences, $\Lmot(p')$ is $n$-truncated if and only if $\Lmot(p)$ is. This gives the desired square.
    
   Let us now verify the uniqueness up to contractible choice. The datum of any square as in the statement yields a factorization of the composite $w \circ p$ into an $\A^1$-equivalence followed by an $\A^1$-fibration. As this is unique up to contractible choice, so is the datum.
\end{proof}
\begin{cor}\label{middle_eq}
    Pulling back along the canonical map $B \to \Lmot(B)$ defines an equivalence of $(n+1)$-categories
    \[\Cov[n]_{\Ao}(\Lmot(B)) \xrightarrow{\sim} \Cov[n]_{\Ao}(B),\]
    whose inverse is given by applying $\Lmot$ objectwise.
\end{cor}
Next we show that, for an $\A^1$-fibration, truncation can be checked before or after applying the motivic localization functor.
\begin{prop}\label{trunc_before_after}
    Let $p \colon E \to B$ be an $\A^1$-fibration. Then the following are equivalent:
    \begin{enumerate}
        \item[(i)] The map $p$ is an $n$-truncated map in $\Shv(\Sm_S)$.
        \item[(ii)] The map $p$ is an $n$-truncated object in $\Cov[\infty]_{\Ao}(B)$.
        \item[(iii)] The map $\Lmot(p)$ is an $n$-truncated map in $\Shv(\Sm_S)$.
    \end{enumerate}
\end{prop}
\begin{proof}
    For the equivalence of the first two statements, note that the inclusion 
    $$\Cov[\infty]_{\Ao}(B) \subset \Shv(\Sm_S)_{/B}$$
    is a fully faithful right adjoint, by the discussion after \cite[\S 1.3]{soa}. In particular, it commutes with the formation of iterated diagonals and it is conservative. Therefore, $p$ is $n$-truncated as an object in $\Cov[\infty]_{\Ao}(B)$ if and only if it is an $n$-truncated map in $\Shv(\Sm_S)$.

    Now we deal with the equivalence of the second and third statement. By our arguments from above, we have that $\Lmot(p)$ is an $n$-truncated map in $\Shv(\Sm_S)$ if and only if it is an $n$-truncated object in $\Cov[\infty]_{\Ao}(\Lmot(B))$. As $p$ is the image of $\Lmot(p)$ under the equivalence 
    $$\Cov[\infty]_{\Ao}(B) \simeq \Cov[\infty]_{\Ao}(\Lmot(B))$$
    from Corollary \ref{middle_eq}, this shows the claim.
\end{proof}
As a preparation for the Galois-theoretic classification, we describe higher $\A^1$-coverings over an $\A^1$-invariant base, including a criterion in terms of the pullback along a section.
\begin{lem}\label{check_on_fib}
    Let $0 \leq n \leq \infty$ and let $p\colon E \to B$ be a map of Nisnevich sheaves where $B$ is $\A^1$-invariant. Then the following are equivalent:
    \begin{enumerate}
        \item[(i)] The map $p$ is an $\A^1$-$n$-covering.
        \item[(ii)] The map $p$ is $n$-truncated and $E$ is $\A^1$-invariant.
    \end{enumerate}
    Suppose furthermore that there exists a section $s$ of the canonical map $B \to B_{\leq 0}$ and that $B_{\leq 0}$ is $\A^1$-invariant. Moreover, let $p'\colon E'\to B_{\leq 0}$ be the pullback of $p$ along $s$. Then the above two statements are also equivalent to:
    \begin{enumerate}
        \item[(iii)] The sheaf $E'$ is $n$-truncated and $\A^1$-invariant.
    \end{enumerate}
\end{lem}
\begin{proof}
    The equivalence of the first two statements is clear if $B \simeq *$. Composing with the map to the point $*$, the general case follows from the closure of $\A^1$-fibrations under 1st-of-3 and from Proposition \ref{trunc_before_after}.
    
    Given the additional assumptions, we now take care of the equivalence of the second and third statement. For this note that by Remark \ref{detect_trunc}, $p$ is $n$-truncated if and only if $p'$ is $n$-truncated, and the latter holds if and only if $E'$ is $n$-truncated, since we can compose with the map to the point as above and $n$-truncated maps are closed under 1st-of-3. Furthermore, the sheaf $E$ is $\A^1$-invariant if and only if $p \in \mathcal{R}_{\A^1}$, where the latter is  the right class of the $\A^1$-modality from Definition \ref{a1_mod}. Since $s$ is an effective epimorphism, \cite[Proposition 3.6.5]{abfj_G} tells us that this happens if and only if $p' \in \mathcal{R}_{\A^1}$. Since $B_{\leq 0}$ is $\A^1$-invariant, this holds if and only if $E'$ is $\A^1$-invariant, where we use 1st-of-3 as above. Hence we are done.
\end{proof}
Combining these preliminary results with Theorem \ref{class_cov_nconn}, we can now prove our next main result.
\begin{thm}\label{classification_of_n_coverings}
    Let $0 \leq n \leq \infty$ and let $B$ be a Nisnevich sheaf. Suppose that there exists a section $s$ of the canonical map $\Lmot(B) \to \underline{\pi}_0^{\A^1}(B)$ and that $\underline{\pi}_0^{\A^1}(B)$ is $\A^1$-invariant. Then applying the functor $\Lmot$ pointwise and taking pullbacks along $s$ afterwards induces an equivalence of $(n+1)$-categories
    \[\Cov[n]_{\Ao}(B) \xrightarrow{\sim} \Mod_{(\Omega_s \Lmot(B))_{\leq n}}\!\left(\left(\Sh_{\A^1}(\Sm_S)_{\leq n}\right)_{/\underline{\pi}_0^{\A^1}(B)}\right),\]
    where $\Omega_s$ denotes the relative loop space in the Nisnevich $\infty$-topos.
\end{thm}
\begin{proof}
    By Corollary \ref{middle_eq}, we may assume that $B$ is equivalent to $\Lmot(B)$. Now, Theorem \ref{class_cov_nconn} applied to $\Shv(\Sm_S)$ yields an equivalence between $\Cov[n](B)$ and $\Mod_{(\Omega_s B)_{\leq n}}\!\left(\left(\Sh(\Sm_S)_{\leq n}\right)_{/B_{\leq 0}}\right)$ by pulling back along $s$. Finally, the third point of Lemma \ref{check_on_fib} tells us that this functor identifies the subcategory $\Cov[n]_{\Ao}(B) \subset \Cov[n](B)$ with modules in the target category whose underlying object is $\A^1$-invariant.
\end{proof}

We conclude this section with a variety of examples of higher $\A^1$-coverings.
\begin{ex}
    Let $X \in \Sm_S$ be an $\A^1$-rigid scheme, so that the canonical map $X \to \Lmot(X)$ is an equivalence. Then a map of smooth schemes $f \colon Y \to X$ is an $\A^1$-$0$-covering if and only if $Y$ is also $\A^1$-rigid. Indeed, the criterion of \cite[Proposition 3.1]{macpherson} tells us that $f$ is an $\A^1$-fibration if and only if the square $f \to \Lmot(f)$ is Cartesian. Since the map $X \to \Lmot(X)$ is an equivalence, this can happen if and only if $Y \to \Lmot(Y)$ is also an equivalence.
    
    As a special case of this example we obtain that the inclusion of any rational point into an $\A^1$-rigid scheme is an $\A^1$-$0$-covering.
\end{ex}
\begin{ex}
    Let $B$ be a pointed Nisnevich sheaf. Let $\Lmot(B)_{\leq n+1, \A^1}$ be the $(n+1)$-truncation of $\Lmot(B)$ in the category of $\A^1$-invariant sheaves. By Lemma \ref{check_on_fib} together with standard truncation estimates, the map $* \to \Lmot(B)_{\leq n+1, \A^1}$ is an $\A^1$-$n$-covering. Therefore, its pullback along the comparison map $B \to \Lmot(B) \to \Lmot(B)_{\leq n+1, \A^1}$ is an $\A^1$-$n$-covering of $B$. 
\end{ex}
\begin{ex}[{\cite[Example 6.5]{macpherson}}]
    Let $n \geq 2$ be an integer and let $G$ be a strictly $\A^1$-invariant Nisnevich sheaf of abelian groups. If the base scheme $S$ is a field, then any $n$-gerbe on a Nisnevich sheaf $B$ banded by $G$ is an $\A^1$-$(n-1)$-covering. The fiber of such a covering at any base point of $B$ is given by $\mathrm{B}^{n-1}G$.
\end{ex}
\begin{ex}
    Let $S$ be a base scheme for which $\mathrm{B}_{\et}\mathbb{G}_m$ is $\A^1$-invariant. Let $X \in \Sm_S$ and consider a line bundle $\mathcal{L}$ on $X$. For any $n > 1$, the $n$-th root stack of $\mathcal{L}$, constructed in \cite[Appendix B, \S B.1]{gw_dm_stacks} as the pullback of the map $(-)^n \colon \mathrm{B}_{\et}\mathbb{G}_m \to \mathrm{B}_{\et}\mathbb{G}_m$ along the classifying map $X \to \mathrm{B}_{\et}\mathbb{G}_m$ of $\mathcal{L}$, is an $\A^1$-$1$-covering of $X$. Since the Kummer sequence is exact in the \'{e}tale topology, the fiber over any base point is given by $\mathrm{B}_{\et}\mu_n$.
\end{ex}
\begin{ex}
    Let $B$ be an $\A^1$-connected $\mathbb{E}_k$-algebra object in Nisnevich sheaves for some integer $k \geq 1$. Then $(\Omega \Lmot(B))_{\leq n}$ is an $\mathbb{E}_{k+1}$-algebra object, which endows the category of modules over $(\Omega \Lmot(B))_{\leq n}$ with an $\mathbb{E}_k$-monoidal structure. By Theorem \ref{classification_of_n_coverings}, this hence also endows the category of $\A^1$-$n$-coverings of $B$ with an $\mathbb{E}_k$-monoidal structure. Translating the formula for the tensor product of modules along the equivalence from Theorem \ref{classification_of_n_coverings}, we can express the tensor product of coverings by means of a bar construction. This method of producing higher $\A^1$-coverings demands further investigation.
\end{ex}
Our final example is the $\A^1$-covering associated to the motivic Hurewicz map. We will study its relation to universal abelian $\A^1$-coverings and apply it to algebraic knots as studied in \cite{ww}, since this was at least partially the inspiration for this article. For this we work over a perfect base field $k$ of characteristic not $2$ and we fix this base scheme for the rest of this subsection.
\begin{ex}
    Let $B$ be a pointed $\A^1$-connected Nisnevich sheaf and let $h_B\colon \underline{\pi}_1^{\A^1}(B) \to \underline{H}_1^{\A^1}(B)$ be the motivic Hurewicz map. There exists then a map
    \begin{center}
    \begin{tikzcd}
    B \arrow[r] & \Lmot(B)_{\leq 1} \cong \mathrm{B}\underline{\pi}_1^{\A^1}(B) \arrow[r, "\mathrm{B}h_B"] & \mathrm{B}\underline{H}_1^{\A^1}(B).
    \end{tikzcd}
    \end{center}
    
    Pulling back the universal covering $* \to \mathrm{B}\underline{H}_1^{\A^1}(B)$ along the above composite, we obtain an $\A^1$-$0$-covering of $B$, which we call the Hurewicz covering of $B$ and we denote it by $C_{\mathscr{H}}(B)$.
\end{ex}
Note that by construction, the Hurewicz covering is functorial in the base $B$.
\begin{lem}\label{hur_cov_is_uni}
    Let $B$ be a pointed $\A^1$-connected Nisnevich sheaf. Then the following are equivalent:
    \begin{enumerate}
        \item[(i)] The Hurewicz covering of $B$ is $\A^1$-connected.
        \item[(ii)] The motivic Hurewicz map $h_B$ is an epimorphism.
    \end{enumerate}
    Furthermore, if any of the above equivalent conditions hold, the Hurewicz covering of $B$ is universal abelian, in the sense of \cite[Definition 4.4]{ww}.
\end{lem}
Before we get to the proof, let us note that point (ii) is not known to hold unconditionally on $B$; see Remark 4.1 of \cite{ww} for a more thorough review. In fact, it is only known for very few examples and over suitable base fields, such as the ones from Appendix B of \cite{ww}.
\begin{proof}
    We consider the fiber sequence
    \begin{center}
    \begin{tikzcd}
    C_{\mathscr{H}}(B) \arrow[r] & B \arrow[r] & \mathrm{B}\underline{H}_1^{\A^1}(B)
    \end{tikzcd}
    \end{center}
    and the long exact sequence of homotopy sheaves
    \begin{center}
    \begin{tikzcd}
    \dotsc \arrow[r] & 0 \arrow[r] & \underline{\pi}_1^{\A^1}(C_{\mathscr{H}}(B)) \arrow[r] & \underline{\pi}_1^{\A^1}(B) \arrow[r, "h_B"] & \underline{H}_1^{\A^1}(B) \arrow[r] & \underline{\pi}_0^{\A^1}(C_{\mathscr{H}}(B)) \arrow[r] & 0 \arrow[r] & 0
    \end{tikzcd}
    \end{center}
    associated with it. Then the claim follows directly from the exactness.
\end{proof}
To apply the above to algebraic knots, we record the behavior of Hurewicz coverings under suitable maps.
\begin{prop}\label{h1_iso_gives_pb}
    Let $f \colon A \to B$ be a map of pointed $\A^1$-connected Nisnevich sheaves which induces an equivalence on $\underline{H}_1^{\A^1}$. Then the square of Nisnevich sheaves
    \[\begin{tikzcd}
    	{C_{\mathscr{H}}(A)} & {C_{\mathscr{H}}(B)} \\
    	A & B
    	\arrow[from=1-1, to=1-2]
    	\arrow[from=1-1, to=2-1]
    	\arrow[from=1-2, to=2-2]
    	\arrow[from=2-1, to=2-2]
    \end{tikzcd}\]
    is Cartesian.
\end{prop}
\begin{proof}
    This square is the back face of the commutative cube
    \[\begin{tikzcd}
    	{C_{\mathscr{H}}(A)} && {C_{\mathscr{H}}(B)} & \\
    	& {*} && {*} \\
    	A && B \\
    	& {\mathrm{B}\underline{H}_1^{\A^1}(A)} && {\mathrm{B}\underline{H}_1^{\A^1}(B)}
    	\arrow[from=1-1, to=1-3]
    	\arrow[from=1-1, to=2-2]
    	\arrow[from=1-1, to=3-1]
    	\arrow[from=1-3, to=2-4]
    	\arrow[from=1-3, to=3-3]
    	\arrow[equals, from=2-2, to=2-4]
    	\arrow[from=2-2, to=4-2]
    	\arrow[from=2-4, to=4-4]
    	\arrow[from=3-1, to=3-3]
    	\arrow[from=3-1, to=4-2]
    	\arrow[from=3-3, to=4-4]
    	\arrow["{\mathrm{B}\underline{H}_1^{\A^1}(f)}"'{pos=0.6}, "\sim"{pos=0.6}, from=4-2, to=4-4]
    \end{tikzcd}\]
    The front face is Cartesian by assumption and the lateral faces are Cartesian by definition. Therefore, so is the back face by pullback pasting.
\end{proof}
Using the above proposition, we can give a slightly modified proof of Theorem 4.8 of \cite{ww}.
\begin{cor}
    Let $X$ be the complement of a closed embedding $\A^1 \hookrightarrow \A^3$ and let $(f,g) \colon X \to \A^2\setminus \{0\}$ be defined as in \cite[Construction 4.3]{ww}. Then the $\underline{\mathrm{K}}^{\mathrm{MW}}_2$-covering $\widetilde{X} \to X$ constructed in loc.\ cit.\ is equivalent to the Hurewicz covering of $X$. In particular,  Lemma \ref{hur_cov_is_uni} yields that it is universal abelian if the Hurewicz map is an epimorphism.
\end{cor}
\begin{proof}
    First note that $X$ is $\A^1$-connected by \cite[Proposition 4.6]{ww} and the proof of loc.\ cit.\ shows that the map $(f,g)\colon X \to \A^2\setminus \{0\}$ induces an equivalence on $\underline{H}_1^{\A^1}$. Moreover, the covering 
    $$\reallywidetilde{\A^2\setminus \{0\}} \to \A^2\setminus \{0\}$$ 
    constructed in \cite[Construction 4.3]{ww} is equivalent to the Hurewicz covering of $\A^2\setminus \{0\}$, as shown in the proof of \cite[Lemma 4.7]{ww}. Therefore, we conclude by using Proposition \ref{h1_iso_gives_pb}.
\end{proof}

\subsection{Counterexample to Morel's suggested geometric criterion for $\A^1$-coverings}
In Remark 7.3 of \cite{morel}, Morel asks whether a morphism of schemes over a field $k$ is an $\A^1$-covering if it is right orthogonal in $\Sm_k$ to the class of all zero sections $0_X \colon X \to \A^1 \times \, X$ where $X \in \Sm_k$. Fix the base field $\C$. The goal of this section is to construct a smooth scheme $T$ together with a morphism $p_1 \colon \Spec(\C) \to T$ which has the above right lifting property, but fails to be an $\A^1$-$0$-covering. In particular, it is a counterexample to Morel's suggested geometric criterion.

\hfill

The construction is based on Construction 4.3 of Balwe, Hogadi and Sawant \cite{BHS}, for which we recall the setup in their notation:

\begin{enumerate}
\item[-] $\lambda_1, \lambda_2 ,\lambda_3 \in \C \setminus \{ 0,1 \}$ are distinct complex numbers
\item[-] $f = (x - \lambda_1)(x - \lambda_2)(x - \lambda_3) \in \C[x]$ is the associated separable polynomial
\item[-] $\alpha_0 = \sqrt{f(0)} = \sqrt{-\lambda_1\lambda_2\lambda_3}$
\item[-] $E = V(y^2 - f) \hookrightarrow \A^2_{x,y}$ is an affine Weierstrass model of the elliptic curve $\overline{E}$ defined by the homogenization of $y^2-f \in \C[x,y]$
\item[-] $\pi \colon E \to \A^1_x$ is the projection onto the $x$-axis
\item[-] $S_1 = V(y^2-t^2f) \hookrightarrow \A^3_{x,y,t}$ is a family of affine plane curves parametrized by $t$ whose fiber over $t = 1$ is given by $E$
\item[-] $p_0 = (0,0,0)$, $p_1 = (0,\alpha_0, 1)$, $q = (1,0,0)$ are three rational points of $S_1$
\item[-] $S = S_1 \setminus \{ p_0 \} = S_1 \setminus \{ 0 \} \hookrightarrow \A^3_{x,y,t}$ (not to be confused with the base scheme from the other sections)
\item[-] $T_1$ is a smooth scheme with the following property. By Lemma 4.4 of \cite{BHS}, there exists a closed embedding $S_1 \hookrightarrow T_1$ into a smooth scheme such that for any field extension $\C \subset L$, if $H \colon \A^1_L \dashrightarrow (T_1)_L$ is a rational map whose image meets $(S_1)_L$, then $H$ factors through the embedding $(S_1)_L \hookrightarrow (T_1)_L$. Note that the construction of $T_1$ shows that it is also separated.
\item[-] $T = T_1 \setminus \{ p_0 \} = T_1 \setminus \{ 0 \}$, which is also separated as an open subscheme of $T_1$
\end{enumerate}
We consider the morphism $\Spec(\C) \to T$ given as the composite of the rational point $p_1 \in S$ and the embedding $S \hookrightarrow T$, which we also denote by $p_1$. Our first goal is to show that this morphism is right orthogonal to all zero sections $0_X \colon X \to \A^1 \times \, X$ where $X$ is a smooth scheme.
\begin{prop}\label{counterex 1}
For any commutative square of the form
\begin{center}
\begin{tikzcd}
X \arrow[r] \arrow[d, "0_X"] & \Spec(\C) \arrow[d, "p_1"] \\
\A^1 \times \,X \arrow[r, "H"] & T,
\end{tikzcd}
\end{center} 
the morphism $H \colon A^1 \times X \to T$ is constant with value $p_1$.
\end{prop}
\begin{proof}
Let $X_i$, $i \in I$, denote the irreducible components of the smooth scheme $X$ with generic points $\eta_i \colon \Spec(L_i) = \Spec(\C(X_i)) \to X$. Base-changing $H$ along the $\eta_i$ we obtain
$$H_i \colon (\A^1_{r})_{L_i} \cong \A^1 \times \, X \times_X \Spec(L_i) \to T_{L_i}$$
with $H_i \circ 0_{L_i} = (p_1)_{L_i}$. 

As a first step, we want to show that $H_i$ is constant with value $p_1$ for any $i \in I$. By definition, the image of the composition of $H_i$ with the inclusion $T_{L_i} \hookrightarrow (T_1)_{L_i}$ meets $(S_1)_{L_i}$ at $(p_1)_{L_i}$. Hence it factors through the inclusion $(S_1)_{Li} \hookrightarrow (T_1)_{L_i}$ by definition of $T_1$. Writing 
$$x_i = H_i^*x,\,\, y_i = H_i^*y \,\, \text{and} \,\, t_i = H_i^*t$$ for the pullbacks of the three coordinate functions in $L_i[r]$, the defining equation of $(S_1)_{L_i}$ yields $y_i^2 = t_i^2f(x_i)$. Moreover, we have $$x_i(0) = 0, \,\, y_i(0) = \alpha_0 \,\, \text{and} \,\, t_i(0) = 1$$
by definition of $H_i$. Therefore $z_i = \frac{y_i}{t_i}$ is defined on the open dense subset $D(t_i) \subset \A^1_{L_i}$. The pair $(x_i,z_i)$ defines a morphism $D(t_i) \to E_{L_i}  = V(y^2 - f) \subset \A^2_{x,y}$ which we extend to a rational map $\varphi \colon \A^1_{L_i} \dashrightarrow E_{L_i}$. By \cite[\href{https://stacks.math.columbia.edu/tag/0BXZ}{Lemma 0BXZ}]{stacks-project}, this rational map extends to a morphism $\overline{\varphi} \colon \P^1_{L_i} \to \overline{E}_{L_i}$ after passing to smooth compactifications. Now we get that $\overline{\varphi}$ is constant by the fact that any morphism from the projective line to a separated group scheme must be constant; see, for instance, \cite[\href{https://stacks.math.columbia.edu/tag/0AEN}{Lemma 0AEN}]{stacks-project}. Since 
$$\overline{\varphi}(0) = \varphi(0) = (x_i(0),z_i(0)) = \biggl(x_i(0),\frac{y_i(0)}{t_i(0)}\biggr) = (0,\alpha_0),$$ the latter is the constant value of $\overline{\varphi}$ and hence also of $\varphi$. In particular we have $x_i = 0$ and $y_i = \alpha_0t_i$ in the function field $L_i(r)$. As those expressions are defined in the subring $L_i[r] \subset L_i(r)$, they already hold there. 

Now we want to observe that $t_i = 1$, which yields that $H_i$ is constant with value $p_1$. For this we consider the closed subscheme $A \subset S_1$ defined by $x = 0$ and $y = \alpha_0t$. Then $H_i \colon \A^1_{L_i} \to (S_1)_{L_i}$ factorizes through $A_{L_i}$ due to the above expressions and we have
$$\mathcal{O}_{A_{L_i}}(A_{L_i}) = L_i[x,y,t]/(y^2-t^2f,x,y-\alpha_0t) \cong L_i[t],$$ 
that is, $A_{L_i} \cong \A^1_{L_i}$ with coordinate $(0,\alpha_0t,t)$. As $H_i$ maps to $T_{L_i}$, we may consider the intersection $(S_1)_{L_i} \times_{(T_1)_{L_i}} T_{L_i} \cong S_{L_i}$ instead of $(S_1)_{L_i}$ and obtain a factorization of $H_i$ through 
$$A_{L_i} \setminus \{ p_0 \} \cong \A^1_{L_i} \setminus \{ 0 \} = (\G_m)_{L_i}.$$ 
Thus $t_i \in L_i[r]^{\times} = L_i^{\times}$ is constant. As $t_i(0) = 1$, we get $t_i = 1$ and hence $y_i = \alpha_0$. Therefore, $H_i$ is constant with value $p_1$.

It remains to observe that $H$ itself is constant with value $p_1$. For this we denote the constant morphism $\A^1 \times \,X \to T$ with value $p_1$ by $\const_{p_1}$ and consider the equalizer $\Eq = \Eq(H,\const_{p_1})$, which can be realized as the pullback $\A^1 \times \, X \times_{T \times T} T$ of $(H,\const_{p_1})$ and the diagonal $\Delta_T$. Since $T$ is separated, the diagonal $\Delta_T$ is a closed immersion, so also the pullback $\Eq \to \A^1 \times \, X$ is a closed immersion. Analogously, base-changing along the generic points $\eta_i$ yields closed immersions
$$\Eq(H_i,(\const_{p_1})_{L_i}) \cong \Eq \times_X \Spec(L_i) \hookrightarrow \A^1 \times \, X \times_X \Spec(L_i) \cong \A^1_{L_i.}$$ 
Since we proved that $H_i = (\const_{p_1})_{L_i}$, these equalizers agree with $\A^1_{L_i}$. In particular, $\Eq$ contains the generic points of all the irreducible components of $\A^1 \times \, X$ and must hence also have the same underlying topological space as $\A^1 \times \, X$. Now let $\mathcal{I} \subset \mathcal{O}_{\A^1 \times X}$ denote the ideal sheaf defining $\Eq$. From the comparison as topological spaces we know that $\sqrt{\mathcal{I}} = 0$, but since $\A^1 \times \, X$ is reduced, this implies $\mathcal{I} = 0$. Thus we have $\Eq = \A^1 \times \, X$ as schemes and hence $H = \const_{p_1}$ as claimed.
\end{proof}
\begin{cor}\label{counterex 2}
The morphism $p_1 \colon \Spec(\C) \to T$ is right orthogonal to the class of zero sections $0_X \colon X \to \A^1 \times \,X$ where $X$ is a smooth scheme.
\end{cor}
\begin{proof}
Since $\Spec(\C)$ is terminal, a filler for a commutative square of the form
\begin{center}
\begin{tikzcd}
X \arrow[r] \arrow[d, "0_X"] & \Spec(\C) \arrow[d, "p_1"] \\
\A^1 \times \,X \arrow[r, "H"] & T,
\end{tikzcd}
\end{center}
can only be the canonical map, which is a filler if and only if $H$ is constant with value $p_1$. Therefore we are done by the previous proposition.
\end{proof}
Next we verify that $p_1 \colon \Spec(\C) \to T$ is not an $\A^1$-$0$-covering. For this we add some further items to the above list:
\begin{enumerate}
\item[-] $\overline{V}_1 = \A^1_x \setminus \{ 0 \} = \G_m$
\item[-] $\overline{V}_2 = (\pi^{-1}(\A^1_x \setminus \{ \lambda_1,\lambda_2,\lambda_3 \})) \setminus \{ (0,-\alpha_0) \}$ is a slightly modified version of $\overline{V}_2$ from \cite{BHS}: without removing the rational point $(0,-\alpha_0)$, two points live over $0$ so that the next bullet point would not be correct. This seems to be a little inaccuracy, but does not affect any of their or any of the following arguments
\item[-] $(\overline{V}_1 \hookrightarrow \A^1_x, \pi \colon \overline{V}_2 \to \A^1_x)$ is an elementary Nisnevich cover
\item[-] This elementary Nisnevich cover together with some additional data forms what Balwe, Hogadi and Sawant call an $\A^1$-ghost homotopy (Convention 4.2 of \cite{BHS}). For our purposes it is not really relevant to know what this is. The only relevant aspect is that it yields an equality
$$[L_{\mot}(p_1)] = [L_{\mot}(q)] \in \pi_0(\Map_{\Shv(\Sm_{\C})}(\Spec(\C),L_{\mot}(S))) = \pi_0^{\A^1}\!(S)(\C),$$
i.e., there is a path $\gamma$ in the above mapping space connecting $L_{\mot}(q)$ with $L_{\mot}(p_1)$. By composing with $L_{\mot}(S \hookrightarrow T)$ the same then holds for the mapping space with $S$ replaced by $T$.
\end{enumerate}
\begin{prop}\label{counterex 3}
The morphism $p_1 \colon \Spec(\C) \to T$ is not an $\A^1$-$0$-covering.
\end{prop}
\begin{proof}
   Since the points $L_{\mot}(p_1)$ and $L_{\mot}(q)$ are homotopic, there exists a homotopy making the square 
   \[\begin{tikzcd}
   	{\Spec(\C) \copro \Spec(\C)} & {\Spec(\C)} \\
   	T & {\Lmot(T)}
   	\arrow[from=1-1, to=1-2]
   	\arrow["{(p_1,q)}"', from=1-1, to=2-1]
   	\arrow["{\Lmot(p_1)}", from=1-2, to=2-2]
   	\arrow[from=2-1, to=2-2]
   \end{tikzcd}\]
   of Nisnevich sheaves commutative. Now suppose that the morphism $p_1 \colon \Spec(\C) \to T$ were an $\A^1$-$0$-covering. Then \cite[Proposition 3.1]{macpherson} tells us that the commutative square of Nisnevich sheaves
   \[\begin{tikzcd}
   	{\Spec(\C)} & \Lmot(\Spec(\C)) \\
   	T & {\Lmot(T)}
   	\arrow[equals, from=1-1, to=1-2]
   	\arrow["{p_1}"', from=1-1, to=2-1]
   	\arrow["{\Lmot(p_1)}", from=1-2, to=2-2]
   	\arrow[from=2-1, to=2-2]
   \end{tikzcd}\]
   would have to be Cartesian. In particular, the morphism $(p_1, q) \colon \Spec(\C) \copro \Spec(\C) \to T$ would need to factor through $p_1 \colon \Spec(\C) \to T$, which is possible if and only if the points $p_1$ and $q$ coincide. This is a contradiction.
\end{proof}
While the above argument gives a rather intuitive understanding of the failure of $p_1$ to be an $\A^1$-$0$-covering, we have an alternative argument including an explicit $\A^1$-equivalence which the map $p_1 \colon \Spec(\C) \to T$ is not right orthogonal to. We thank Nidhi Gupta for suggesting to look for such an example.
\begin{rem}
    We set $\overline{V}_{12} = \overline{V}_{1} \times_{\A^1_x} \overline{V}_{2}$. For this example we need to use the $\A^1$-ghost homotopy between the points $p_1$ and $q$ from Construction 4.3 of \cite{BHS} a bit more explicitly: it defines a map from the colimit $C$ of the diagram of Nisnevich sheaves
    \begin{center}
    \begin{tikzcd}
   	\overline{V}_1 & \overline{V}_{12} \arrow[l] \arrow[r, "0"] & \A^1_{\overline{V}_{12}} & \overline{V}_{12} \arrow[l, swap, "1"] \arrow[r] & \overline{V}_{2}
    \end{tikzcd}
    \end{center}
    to $T$. The point $(0, \alpha_0) \colon \Spec(\C) \to \overline{V}_2$ then gives us a map $\Spec(\C) \to C$ by composition, whose further composite with the map $C \to T$ from above picks out the point $p_1 \in T$. Therefore, the square
    \[\begin{tikzcd}
	    \Spec(\C) & \Spec(\C) \\
	    C & T
	    \arrow[equals, from=1-1, to=1-2]
	    \arrow[from=1-1, to=2-1]
	    \arrow["{p_1}", from=1-2, to=2-2]
	    \arrow[from=2-1, to=2-2]
    \end{tikzcd}\]
    commutes and does not admit a diagonal filler, as the map $C \to T$ is not constant. 
    
    To conclude, it suffices to show that the map $\Spec(\C) \to C$ is an $\A^1$-equivalence. To that end, we first consider the map $C \to \A^1_{\C}$ given by mapping the defining diagram of $C$ to the diagram
    \begin{center}
    \begin{tikzcd}
   	\overline{V}_1 & \overline{V}_{12} \arrow[r, equals] \arrow[l] & \overline{V}_{12} & \overline{V}_{12} \arrow[l, equals] \arrow[r] & \overline{V}_{2}
    \end{tikzcd}
    \end{center}
    and then taking colimits. As a colimit of $\A^1$-equivalences, it is an $\A^1$-equivalence itself. By $2$-out-of-$3$ for $\A^1$-equivalences, the sequence of maps $\Spec(\C) \to C \to \A^1_{\C}$ thus yields the claim.
\end{rem}

Combining Corollary \ref{counterex 2} with Proposition \ref{counterex 3}, we obtain a counterexample to Morel's suggested geometric criterion for $\A^1$-$0$-coverings from Remark 7.3 of \cite{morel}:
\begin{thm}\label{counterex to Morel}
The morphism $p_1 \colon \Spec(\C) \to T$ is right orthogonal to the class of all zero sections $0_X \colon X \to \A^1 \times \, X$ where $X$ is a smooth $\C$-scheme, but it is not an $\A^1$-$0$-covering.
\end{thm}
 This result shows that right lifting against $\A^1$-equivalences and right lifting against zero sections behave differently with respect to algebraic homotopies and $\A^1$-ghost homotopies in general. In particular, we can also not expect Morel's geometric criterion to have a direct extension for higher $\A^1$-coverings.

\subsection{Morel's geometric criterion for $\A^1$-coverings of presheaves}
In this section we show that the Nisnevich topology is the real culprit of the failure from the previous section. For this we reformulate all the relevant definitions for presheaves instead of sheaves and then show that the geometric criterion from Remark 7.3 of \cite{morel} is correct in this setting.

\hfill

We start by recalling some classical facts about the $\A^1$-singular complex functor. In what follows, we fix an arbitrary base scheme $S$ and we denote the mapping presheaf in $\PShv(\Sm_S)$ by $[-,-]$.
\begin{de}
    Let $f\colon X \to Y$ be a map of Nisnevich presheaves. We call $f$ a pre-$\A^1$-equivalence if the map of spaces $\Map_{\PShv(\Sm_S)}(f, Z)$ is an equivalence for any $\A^1$-invariant presheaf $Z$. For $0 \leq n \leq \infty$, we call a map of presheaves $p \colon E \to B$ a pre-$\A^1$-$n$-covering if it is $n$-truncated and right orthogonal to all pre-$\A^1$-equivalences.
\end{de}
As in the setting of sheaves, we will refer to a pre-$\A^1$-$\infty$-covering as a pre-$\A^1$-fibration. 
\begin{con}[{\cite[p. 88]{mv}}]\label{lmot_def}
    Let us denote by $c \colon \A^1 \to *$ the constant map to the point and by $i_0, i_1\colon * \to \A^1$ the sections of $\A^1$ at $0$ and $1$. Define $\A^\bullet\colon \Delta \to \Sm_{\Z}$ to be the cosimplicial scheme whose $n$-th object is $\A^n$, whose coface maps are given by
    \begin{align*}
        \A^{n-1} \xrightarrow{\A^\bullet(\delta^n_k)} \A^n & = \begin{cases}
                i_0 \times \id^{\times n-1}, & k = 0 \\
                \id^{\times k-1} \times \, \mathrm{diag} \times \id^{\times n-k-1}, & 0 < k < n \\
                \id^{\times n-1} \times \, i_1, & k = n\end{cases}
    \end{align*}
    and whose codegeneracy maps are given by
    \begin{align*}
        \A^{n+1} \xrightarrow{\A^\bullet(\sigma^n_k)} \A^n & =  \id^{\times k} \times \, c \times \id^{\times n-k}.
    \end{align*}
    Let $\Sing^{\A^1} \colon \PShv(\Sm_S) \to \PShv(\Sm_S)$ be the functor sending an object $X$ to the geometric realization $\lvert [\A^\bullet \times \,S, X] \rvert$. The comparison maps from $\left[\A^0\times \, S, X\right] = X$ to the colimit assemble into a natural transformation $\id \to \Sing^{\A^1}$. By the same argument as in \cite[Theorem 4.25]{primer}, the functor $\Sing^{\A^1}$ is the Bousfield localization functor onto the subcategory of $\A^1$-invariant presheaves.
\end{con}
To analyze the lifting condition against zero sections, we first reformulate orthogonality in terms of mapping presheaves. The following is a direct application of Lemma 2.4.3 from \cite{abfj}.
\begin{lem}\label{lift_htpy}
    Let $f\colon X \to Y$ and $g\colon A \to B$ be maps of Nisnevich presheaves. Then the following are equivalent:
    \begin{enumerate}
        \item[(i)] The map $f$ is right orthogonal to $g \times T$ for all $T \in \Sm_S$.
        \item[(ii)] The square $g^* \colon [B, f] \to [A, f]$ is Cartesian in $\PShv(\Sm_S)$.
    \end{enumerate}
\end{lem}
This reformulation allows us to link the relevant lifting condition to the $\A^1$-singular complex.
\begin{prop}\label{a1_geom}
    Let $f\colon X \to Y$ be a map of Nisnevich presheaves which is right orthogonal to all zero sections $T \to \A^1\times \, T$ where $T \in \Sm_S$. Then the square $f \to \Sing^{\A^1}(f)$ is Cartesian in $\PShv(\Sm_S)$.
\end{prop}
\begin{proof}
Using descent for the $\infty$-topos $\PShv(\Sm_S)$, the claim follows if the square 
$$[\A^{j} \times S, f] \to [\A^{i} \times \, S, f]$$
is Cartesian for any map $[i] \to [j]$ in the simplex category $\Delta$. Moreover, since any object in $\Delta$ admits at least one map from $[0]$, pullback pasting allows us to reduce to the case $i = 0$. Let us hence deal with this case. By Construction \ref{lmot_def}, the relevant square is induced by a map $\iota \colon S \to \A^j \times \, S$. Since $f$ is right orthogonal to all zero sections, it is right orthogonal to $\iota \times T$ for $T \in \Sm_S$ by induction on $j$. Thus we can then apply Lemma \ref{lift_htpy} and we are done.
\end{proof}
\begin{cor}\label{morels_question}
    A map of Nisnevich presheaves $X \to Y$ (respectively a map $X \to Y$ between smooth schemes) is a pre-$\A^1$-fibration (respectively a pre-$\A^1$-$0$-covering) if and only if it is right orthogonal to zero sections $T \to \A^1 \times \, T$ where $T$ is a smooth scheme over $S$.
\end{cor}
\begin{proof}
    The non-parenthesized statement follows from Proposition \ref{a1_geom} together with \cite[Proposition 3.1]{macpherson}. The parenthesized statement follows from the Yoneda lemma and the fact that maps of schemes automatically yield $0$-truncated maps of presheaves.
\end{proof}
\printbibliography
\end{document}